\documentclass[reqno]{amsart}
\newcommand \datum {May 4, 2022}

\usepackage{amssymb,latexsym}
\usepackage{amsmath}
\usepackage{amscd}
\usepackage{graphicx}
 \usepackage{color}
\usepackage{enumerate}
\numberwithin{equation}{section}
\theoremstyle{plain}
 \newtheorem{theorem}{Theorem}[section]
 \newtheorem{lemma}[theorem]{Lemma}

 \newtheorem{remark}[theorem]{Remark}
\theoremstyle{definition}

\theoremstyle{remark}
 \newtheorem{case}{Case}

\newcommand \CDA {\textup{CDA}}
\newcommand \LAT {\textup{LAT}}
\newcommand \CE [1] {\textup{CE}(#1)}

\newcommand \En [1] {\textup{En}(#1)}
\newcommand \At [1] {\textup{At}(#1)}
\newcommand \nb [2] {|#2/#1|} 
\newcommand \Con [1]   {\textup{Con}(#1)}
\newcommand \Equ [1]   {\textup{Equ}(#1)}
\newcommand \vonal {\noalign{\hrule}}

\newcommand \heq {h_{\textup{eq}}}
\newcommand \gmx {g_{\textup{mx}}}
\newcommand \gsb {g_{\textup{sb}}}
\newcommand \gpn {g_{\textup{pn}}}
\newcommand \gplu {+_{\textup{glu}}}
\newcommand \CA {\textup{CA}}
\newcommand \CB {\textup{CB}}
\newcommand \falf {f_{\alpha}}
\newcommand \fabe {f_{\beta}}
\newcommand \faf {f_{a}}
\newcommand \botof[1] {\Delta{\kern-1pt}_{#1}}
\newcommand \topof[1] {\nabla{\kern-2pt}_{#1}}
\newcommand \equ {\textup{equ}}
\newcommand \con {\textup{con}}
\newcommand \ideal [1] {\mathord{\downarrow}#1}
\newcommand \filter [1] {\mathord{\uparrow}#1}
\newcommand \downset [1]{\mathcal S_{\textup{dn}}(#1)}
\newcommand \peup {\twoheadrightarrow_{\textup{up}}}
\newcommand \pedn {\twoheadrightarrow_{\textup{dn}}}
\newcommand \pers {\twoheadrightarrow}
\newcommand \proj {\mathrel{\twoheadrightarrow^{\kern-1pt\ast}}}
\newcommand \eqflat {\mathrel{=\kern-5pt=}}
\newcommand \Jir [1] {\textup J(#1)} 
\newcommand \tbf[1]  {\textbf{#1}} 
\newcommand \set [1]{\{#1\}}

\newcommand\red[1]{{\textcolor{red}{#1}}}
\newcommand \blue [1] {{\color{blue}#1\color{black}}}
\newcommand \black [1] {{\color{black}#1\color{black}}}
\newcommand \magenta [1] {{\color{magenta}#1\color{black}}}

\begin{document}

\title{Lattices with lots of congruence energy} 

\author[G.\ Cz\'edli]{G\'abor Cz\'edli}
\email{czedli@math.u-szeged.hu}
\urladdr{http://www.math.u-szeged.hu/~czedli/}
\address{ Bolyai Institute, University of Szeged, Hungary}

\begin{abstract} 
In 1978, motivated by E. H\"uckel's work in quantum chemistry, I.~Gutman introduced the concept of the \emph{energy}  of a finite simple graph $G$  as the sum of the absolute values of the eigenvalues of the adjacency matrix of $G$.
At the time of writing, the MathSciNet search for "Title=(graph energy) AND  Review Text=(eigenvalue)" returns 351 publications, most of which going after Gutman's definition. 
A congruence $\alpha$ of a finite algebra $A$ turns $A$ into a simple graph: we connect $x\neq y\in A$   by an edge iff $(x,y)\in\alpha$; we let $\En\alpha$ be the energy of this graph. 
We introduce the \emph{congruence energy} $\CE A$ of $A$ by $\CE A:=\sum\{\En\alpha: \alpha\in\Con A\}$. 
Let $\LAT(n)$ and $\CDA(n)$ stand for the class of $n$-element lattices and that of $n$-element congruence distributive algebras of any type. For a class $\mathcal X$, let
$\CE{\mathcal X}:=\{\CE A: A\in \mathcal X\}$.
We prove the following.
(1) For $\alpha\in A$, $\En\alpha/2$ is  the height of $\alpha$ in the \emph{equivalence} lattice of $A$.  (2) The largest number and the second largest number in $\CE{\LAT(n)}$ are $(n-1)\cdot 2^{n-1}$ and, for $n\geq 4$, $(n-1)\cdot 2^{n-2}+2^{n-3}$; 
these numbers are only witnessed by chains and lattices with exactly one two-element antichain, respectively. (3) The largest number in $\CE{\CDA(n)}$ is  also $(n-1)\cdot 2^{n-1}$, and if $\CE A=(n-1)\cdot 2^{n-1}$ for an $A\in\CDA(n)$, then $\Con A$ is 
a boolean lattice with size $|\Con A|=2^{n-1}$.
\end{abstract}

\thanks{This research of the first author was supported by the National Research, Development and Innovation Fund of Hungary under funding scheme K 138892.}



\date{\datum
\hfill 
$\angle\kern-5.5pt\text{\textunderscore} 
\kern-2pt\text{\textunderscore} 
\kern-1.5pt\text{\textunderscore} 
\kern-1.0pt\text{\textunderscore} 
\kern-7.0pt\text{\raisebox{1.7pt}{\rotatebox{135}{$\smallsmile$}}}
\kern-5.1pt\text{\tbf{\raisebox{1.9pt}.}}
$
\raisebox{2pt}
{\tiny{\texttt{http://www.math.u-szeged.hu/\textasciitilde{}czedli/}}}\qquad
}

\maketitle

\section{Targeted readership}\label{sect:intro}
Most mathematicians are expected to read the \emph{results} of this paper easily. These 
result might motivate analogous investigation of some 
algebraic structures not mentioned here. To follow the \emph{proofs},  a little familiarity with lattice theory is assumed.

\section{Outline}\label{subs:outl}
Sections~\ref{subs:hist} and \ref{subsect:morehistory}  give some history and motivations. Section \ref{subsect:keyc} introduces the key concepts.
Section~\ref{sect:obs} translates these concepts from linear algebra to  lattice theory. Section~\ref{sect:mainresult} states the main result of the paper, Theorem~\ref{thmmain}. Section~\ref{sect:proof}, which comprises the majority of the paper, proves the main result.

\section{Motivations coming from quantum chemistry and graph theory}\label{subs:hist}
The research aiming at the energy of a graph goes back to H\"uckel~\cite{huckel},  which is a quantum chemical paper published more then nine decades ago. It would be difficult to present a short survey of how the research of the energy of an unsaturated conjugated hydrocarbon molecule lead to a concept on the border line between graph theory and linear algebra. Thus, if the reader is interested in these historical details, then he is referred to the introductory part of 
Majstorovi\'c,   Klobu\v car, and Gutman~\cite
{majst-at-al}. What is important for us is that Gutman's pioneering paper \cite{gutman} introduced the concept of the energy of a graph in 1978, and this 
concept has been studied in quite many publications since then.

The concept of energy can be extended to mathematical structures that are accompanied by graphs. This is exemplified by Pawar and Bhamre~\cite{pawarbhamre}. As opposed to Gutman~\cite
{majst-at-al}, Pawar and Bhamre~\cite{pawarbhamre} use non-simple graphs. Here we stick to simple graphs but we need a family of them.

\section{The concept we introduce}\label{subsect:keyc} 
A \emph{simple graph} is an undirected graph without loop edges and multiple edges. 
Let $v_1,\dots,v_n$ be a repetition-free list of all vertices of a finite simple graph $G$. (The order of the vertices in this list will turn out to be unimportant in \eqref{eq:EnGwelldef} later.)
The \emph{adjacency matrix} of $G$ is the 
$n$-by-$n$ matrix $B=(a_{ij})_{n\times n}$ with entries 0 and 1 according to the rule that $a_{ij}=1$ iff $v_i$ and $v_j$ are connected by an edge. 
The $n$-by-$n$ unit matrix is denoted by $I_n$; every diagonal entry of $I_n$ is 1 while any other entry of $I_n$ is 0.
Since $B$ is a symmetric matrix, its characteristic
polynomial is known to be the product of linear factors over $\mathbb R$, that is, $\det(xI_n-B)=\prod_{j:=1}^n(x-x_j)$ with real numbers (called eigenvalues) $x_1,\dots, x_n$. According to Gutman  \cite{gutman},
the energy of the graph $G$ in question is defined to be $\En G:=\sum_{j=1}^n|x_j|$. Note that
\begin{equation}\left.
\parbox{10.6cm}{if we change the order of elements in the list $v_1,\dots,v_n$, then $B$ turns into another matrix $B'$; however, $B$ and $B'$ are similar matrices with the same characteristic polynomial, whereby $\En G$ is well defined.}\,\,\right\}
\label{eq:EnGwelldef}
\end{equation}
Next, we start from a finite algebra $A=(A,F)$. A congruence $\alpha$ of $A$, in notation, $\alpha\in\Con A$, determines a graph $G_{A,\alpha}$ in quite a natural way: the vertices are the elements of $A$ while $a,b\in A$ are connected by an edge of $G$ iff $a\neq b$ but $(a,b)\in\alpha$. We define \emph{the energy} $\En\alpha$ \emph{of the congruence $\alpha$} by letting $\En\alpha:=\En{G_{A,\alpha}}$.

Note that $\En\alpha$ is meaningful for any equivalence relation $\alpha$ of $A$, in notation, $\alpha\in\Equ A$ since $\Equ A=\Con{A,\emptyset}$ (the case of no operation). In fact, $\En\alpha$ is meaningful for any symmetric relation $\alpha$ of $A$ but in this paper we restrict ourselves to congruence relations.
To explain the definition of $\En\alpha$ more directly  and for the sake of later reference, assume that $A=\set{a_1,\dots,a_n}$ is an $n$-element algebra and
$\alpha\in\Con A$. Let $\botof A:=\set{(x,x): x\in A}$ denote the smallest congruence of $A$.
Up to matrix similarity, the \emph{adjacency matrix} of $\alpha$ is 
\begin{equation}
M(\alpha)=(m_{i j})_{n\times n}\,\,\text{ where }\,\,
m_{i j}=
 \begin{cases}
  1,&\text{if }(a_i,a_j)\in \alpha\setminus\botof A,\cr
  0,&\text{ otherwise.}
 \end{cases}
\end{equation}
Then the characteristic polynomial of $M(\alpha)$ is  $\chi_{M(\alpha)}=\prod_{j:=1}^n(x-x_j)$. 
Keeping \eqref{eq:EnGwelldef} in mind, we define the 
\red{$\boxed{\text{
                   \black{\emph{energy} $\En\alpha$ of $\alpha$ by $\En\alpha:=\sum_{j=1}^n|x_j]$.
                         }
                  }
            }
     $
    }

If we take all congruences $\alpha$ of $A$ and form the sum of their  $\En\alpha$'s, then we obtain the \emph{congruence energy} $\CE A$ of our algebra. So the key definition in the paper is the following: for a finite algebra $A$, 
\begin{equation}\magenta{
\boxed{\blue{
\boxed{\red{
\boxed{\black{\text{the \emph{congruence energy} $\CE A$ of $A$ is $\,\,\CE A:=\sum_{\alpha\in\Con A} \En \alpha$.}}}}}}}}
\label{eq:EnergofA}
\end{equation}

\section{Motivations coming from algebra}\label{subsect:morehistory} 
A straightforward way to measure the complexity of the collection of congruences of a finite algebra $A$ is to take  $|\Con A|$; $\CE A$ offers another way. Figure \ref{fig:ex} shows that none of the inequalities  $\CE{A_1}<\CE{A_2}$ and $|\Con{A_1}|< |\Con{A_2}|$ implies the other one. Among the $n$-element algebras $A$, those with \emph{minimal}  $|\Con A|$  could be the involved \emph{the building stones} of other algebras, like finite simple groups. On the other hand, $n$-element algebras $A$ with \emph{maximal} or close to maximal  $|\Con A|$ are often 
\emph{nice buildings} with well-understood structures and nice properties; see, for example, Cz\'edli~\cite{czedli145,czedli148,czedli152} and Kulin and Mure\c san \cite{kulinmuresan}. In addition to Section \ref{subs:hist}, these ideas also motivate the present paper.

\section{Two easy remarks}\label{sect:obs}

For a finite algebra $A$ and  $\alpha\in \Con A$, the quotient algebra $A/\alpha$ consists of the $\alpha$-blocks, whereby $\nb\alpha A$ is the \emph{number of the blocks of }$\alpha$. We will denote by $\Equ A=(\Equ A,\subseteq)$ the \emph{equivalence lattice} of $A$; note that  $\Con A$ is a sublattice of $\Equ A$
containing the least equivalence  $\botof A=\set{(x,x): x\in A}$ and the largest equivalence $\topof A=A\times A$.  The covering relation understood in $\Equ A$ is denoted by $\prec_e$. For $\alpha\in \Con A$, the 
\emph{height} of $\alpha$ in $\Equ A$ will be denoted by $\heq (\alpha)$. In particular, $\heq(\botof A)=0$ and $\heq(\topof A)=|A|-1$. Using the semimodularity of $\Equ A$, see, for example,
 Gr\"atzer~\cite[Theorem 404]{gratzerLTFbook}, we obtain trivially that
\begin{equation}
\text{for $\alpha<\beta$ in $\Equ A$,   $\,\,\alpha\prec_e \beta$ if and only if $\nb\beta A=\nb\alpha A-1$.}
\label{eqtxt:pRnwq}
\end{equation}
It follows from \eqref{eqtxt:pRnwq} that
\begin{equation}
\text{if $|A|=n$ and $\alpha\in\Equ A$, then $\heq(\alpha)+\nb\alpha A = n$.}
\label{eq:heqnb}
\end{equation}

\begin{remark}\label{ob:dsnZdlsmnhs} 
For an $n$-element finite algebra $A$ and $\Theta\in\Con A$,
\begin{align}
\En \Theta&=2\cdot(n-\nb\Theta A)=2\cdot\heq(\Theta)\qquad \text{ and  }
\label{eq:shmpwWr}\\
\CE A&=2n\cdot |\Con A|-2\cdot\sum_{\Theta\in\Con A}\nb \Theta A = 2\cdot \sum_{\alpha\in\Con A} \heq(\Theta).
\label{eq:EnThdqm}
\end{align}
\end{remark}

\begin{proof}
Consider the following $k$-by-$k$ matrices:
\begin{equation*}
M_k:=\begin{pmatrix}
0&1&1& \dots &  1\cr
1&0&1& \dots &  1\cr
1&1&0& \dots &  1\cr
\vdots&\vdots&\vdots&\ddots&\vdots\cr
1&1&1& \dots &  0\cr
\end{pmatrix},
\quad
P_k:=\begin{pmatrix}
-1&-1& \dots & -1& 1\cr
1&0& \dots & 0& 1\cr
0&1& \dots & 0& 1\cr
\vdots&\vdots&\ddots&\vdots&\vdots\cr
0&0& \dots & 1& 1\cr
\end{pmatrix},
\end{equation*}
\begin{equation*}
Q_k:=\begin{pmatrix}
-1&k-1&-1& \dots &  -1\cr
-1&-1&k-1& \dots &  -1\cr
\vdots&\vdots&\vdots&\ddots&\vdots\cr
-1&-1&-1& \dots &  k-1\cr
1&1&1& \dots &  1\cr
\end{pmatrix},
\quad
H_k:=\begin{pmatrix}
-1&0& \dots & 0& 0\cr
0&-1& \dots & 0& 0\cr
\vdots&\vdots&\ddots&\vdots&\vdots\cr
0&0& \dots & -1& 0\cr
0&0& \dots & 0& k-1\cr
\end{pmatrix}.
\end{equation*}
Note that each of $P_k$, $Q_k$, and $H_k$ contains a $(k-1)$-by-$(k-1)$ submatrix in which the diagonal elements are all equal and so do the non-diagonal elements; these submatrices are the bottom left $(k-1)$-by-$(k-1)$ submatrix of $P_k$, the top right one of $Q_k$, and the top left one of $H_k$. An easy computation shows that 
$P_kQ_k=kI_k$, implying that $P_k^{-1}= k^{-1}Q$. Another computation yields that $P_kH_kQ_k=kM_k$, whereby $M_k= P_kH_k (k^{-1}Q_k)=  P_kH_k P_k^{-1}$. 
This shows that $M_k$ and $H_k$ are similar matrices with the same characteristic polynomial and eigenvalues.  Hence, the sum of the absolute values of the eigenvalues of $M_k$ is $2(k-1)$. 
Let $U_1,\dots,U_t$ be the $\Theta$-blocks where $t=\nb\Theta A$. For $i=1,\dots,t$, let $k_i:=|U_i|$.
List the elements of $A$ so that first we list the elements of $U_1$, then the elements of $U_2$, and so on. Then $M(\Theta)$ is the matrix we obtain by placing $M_{k_1}, \dots, M_{k_t}$ along the diagonal and putting zeros everywhere else. Then the system of the eigenvalues of $M(\Theta)$ is the union of the systems of the eigenvalues of the $M_{k_i}$, $i=1,\dots,t$. Thus, using that $k_1+\dots+k_t=|A|=n$, 
$\En\Theta = 2(k_1-1)+\dots+ 2(k_t-1)= 2n-2t=2(n-\nb\Theta A)$, as required. The rest of Remark \ref{ob:dsnZdlsmnhs}  is now trivial. 
\end{proof}

For a positive integer $n$, let $[n]:=\set{1,2,\dots,n}$, and let  $k\in[1]$. Recall that
$B(n):=|\Equ{[n]}|$ is the $n$-th \emph{Bell number}, $S_2(n,k):=|\set{\alpha\in\Equ{[n]}: \nb\alpha {[n]}=k}|$ is a \emph{Stirling number of the second kind}, and $
B_2(n):=\sum_{i=1}^n i\cdot S_2(n,i)$ is the $n$-th \emph{$2$-Bell number}.
They are frequently studied numbers; see the sequences A000110, A008277, and A005493 and  A138378 in Sloan's OEIS  \cite{sloan}. 

\begin{remark} For a positive integer $n$ and an $n$-element algebra $A$, we have that 
\begin{equation}
\CE A\leq 2n B(n)-2B_2(n).
\label{eq:smwTrzSg}
\end{equation}
In \eqref{eq:smwTrzSg}, equality holds if and only if $\Con A=\Equ A$. 
\end{remark}
The straightforward details of the proof are omitted.
Let $\eqref{eq:smwTrzSg}(n)$ stand for $2nB(n)-2B_2(n)$; the first ten values of  $\eqref{eq:smwTrzSg}(n)$ are given in the following table.
\[
\lower  0.8 cm
\vbox{\tabskip=0pt\offinterlineskip 
\halign{\strut#&\vrule#\tabskip=1pt plus 2pt&#\hfill& \vrule\vrule\vrule#&
\hfill#&\vrule#&
\hfill#&\vrule#&
\hfill#&\vrule#&
\hfill#&\vrule#&
\hfill#&\vrule\tabskip=0.1pt#&#\hfill\vrule\vrule\cr
\vonal\vonal\vonal\vonal
&&\hfill$n$&&$\,1$&&$\,2$&&$\,3$&&$\,4$&&$5$&\cr\vonal
&&$\eqref{eq:smwTrzSg}(n)$&&$0$&&$2$&&$10$&&$46$&&$218$&\cr\vonal\vonal\vonal\vonal
&&\hfill$n$&&$\,6$&&$7$&&$8$&&$9$&&$10$&\cr\vonal
&&$\eqref{eq:smwTrzSg}(n)$&&$1\,088$&&$5\,752$&&$32\,226$&&$190\,990$&&$1\,194\,310$&\cr\vonal\vonal\vonal\vonal
}} 
\]

\section{The main result}\label{sect:mainresult}
An algebra $A$ is \emph{congruence distributive} if the lattice $\Con A=(\Con A,\subseteq)$ is distributive. Lattices are congruence distributive. \emph{Chains} are lattices in which any two elements $x$ and $y$ are comparable, in notation, $x\nparallel y$. The $n$-element chain is denoted by $C_n$, and let $B_4$ be the 4-element boolean lattice. The \emph{glued sum} $U\gplu V$ of disjoint finite lattices $U$ and $V$ is $(U\cup (V\setminus\set{0_V}),\leq)$
where $x\leq y$ iff $x\leq_U y$, $x\leq_V y$, or $(x,y)\in U\times V$. Note that the $U\gplu V$ is a particular case of Hall--Dilworth gluing. 
In order to formulate the main result of the paper, we define
\begin{equation}
\gmx(n):= (n-1)\cdot 2^{n-1}\,\,\text{  and } \,\,
\gsb(n):=(n-1)\cdot 2^{n-2} + 2^{n-3};
\label{eq:njpdFwmtS}
\end{equation}
The acronyms in the subscripts come from 
``\raisebox{-.3ex}{M}a\raisebox{-.3ex}{X}imal'' and ``\raisebox{-.3ex}{S}u\raisebox{-.3ex}{B}maximal'.

\begin{theorem}\label{thmmain} For any positive integer  $n$, the following three assertions hold. 

\textup{(a)} Let $A$ be an $n$-element congruence distributive algebra. Then we have that
$\CE A \leq \gmx(n)$. Furthermore, if $\,\CE A = \gmx(n)$, then $\Con A$ is a boolean lattice and $|\Con A|=2^{n-1}$.

\textup{(b)} Let $L$ be an $n$-element lattice. Then $\CE L \leq \gmx(n)$. Furthermore, 
$\CE L = \gmx(n)$ if and only if $L$ is the $n$-element chain.

\textup{(c)} Let $L$ be an $n$-element lattice such that $\CE L<\gmx(n)$. Then $n\geq 4$ and   $\CE L\leq\gsb(n)$.
Furthermore, $\CE L=\gsb(n)$ if and only if there is exactly one $2$-element antichain in $L$. Equivalently, 
$\CE L=\gsb(n)$ if and only if 
there are finite chains $C'$ and $C''$ such that $L=C'\gplu B_4\gplu C''$.
\end{theorem}

\section{Proving Theorem~\ref{thmmain}}\label{sect:proof}

To prove the theorem, we need  several preparatory statements. For elements $x$ and $y$ of a lattice $L$, the least congruence $\Theta\in\Con L$ containing $(x,y)$ is denoted by $\con(x,y)$. Similarly, $\equ(x,y)$ stands for the least equivalence relation containing the pair $(x,y)$. An element $a\in L$ is an \emph{atom} if 
$0\prec a$. The set of atoms of a lattice $L$ will be denoted by $\At L$. 
The following result is a counterpart of Theorem~\ref{thmmain}.

\begin{lemma}[Cz\'edli \cite{czedli145}]\label{lemma:manycon}
 Let $n$ be  a positive integer. For an $n$-element lattice $L$ and an $n$-element congruence distributive algebra $A$, the following hold.

\textup{(a)} $|\Con A|\leq 2^{n-1}$. Also, if $\,|\Con A| = 2^{n-1}$, then $\Con A$ is a boolean lattice.

\textup{(b)} $|\Con L|\leq 2^{n-1}$.  Furthermore, $|\Con L| =  2^{n-1}$ if and only if $L$ is a chain.

\textup{(c)} If  $|\Con L| < 2^{n-1}$, then $|\Con L|\leq 2^{n-2}$. Also, $|\Con L|= 2^{n-2}$ if and only if there are finite chains $C'$ and $C''$ such that $L=C'\gplu B_4\gplu C''$.
\end{lemma}

Prior to Cz\'edli \cite{czedli145},  part (b) of this lemma was proved by Freese~\cite{freese}. Theorem~\ref{thmmain} needs a more involved proof than Lemma~\ref{lemma:manycon}; Figure \ref{fig:ex} allows us to guess why.

\begin{figure}[ht]
\vspace{-1em}
\centerline
{\includegraphics[scale=0.9]{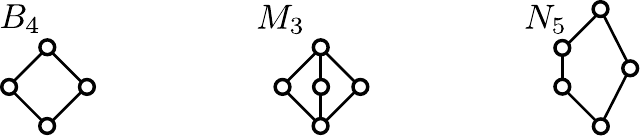}}
\vspace{-0.5cm}
\caption{$B_4$, $M_3$, and $N_5$}\label{figfadsom}
\label{fig:M3N5}
\vspace{-1em}
\end{figure}

\emph{Up congruence perspectivity} and \emph{down congruence perspectivity} will be denoted by $\peup$ and $\pedn$, respectively. That is, for intervals $[a,b]$ and $[c,d]$ of a lattice $L$, $[a,b]\peup[c,d]$ means that 
$b\vee c=d$ and $a\leq c$ while $[a,b]\pedn[c,d]$ stand for the conjunction of $a\wedge d=c$ and $b\geq d$.
Congruence perspectivity and \emph{congruence projectivity}
are denoted by $\pers$ and $\proj$, respectively; 
$[a,b]\pers[c,d]$ means that $[a,b]\peup[c,d]$ or  $[a,b]\pedn[c,d]$ while $\proj$ is the transitive and reflexive closure of $\pers$. An interval $[a,b]$ is \emph{prime} 
if $a\prec b$. 
The least element and the largest element of an interval $I$ are denoted by $0_I$ and $1_I$, respectively. Except for its part (1), the following lemma belongs to the folklore.

\begin{lemma}\label{lemma:aboutcon}
Let $L$ be a \emph{finite} lattice. Then the following assertions hold.

\textup{(1)} By Gr\"atzer~\cite{gratzTechn},  
an $\alpha\in\Equ L$ is a congruence of $L$ if and only if the $\alpha$-blocks are intervals and for any $x,y,z\in L$ the following implication and its dual hold:
\begin{equation}
\text{if $x\prec y$, $x\prec z$, and $(x,y)\in\alpha$, then $(z,y\vee z)\in\alpha$.}
\end{equation}

\textup{(2)} $\Jir{\Con L}=\set{\con(a,b):a\prec b}$.  
Consequently, a congruence is determined by the prime intervals it collapses.

\textup{(3)} For prime intervals   $[a,b]$ and $[c,d]$ 
of $L$, 
\begin{equation}
(c,d)\in\con(a,b) \iff \con(a,b)\geq\con(c,d)  \iff [a,b]\proj[c,d].
\end{equation}

\textup{(4)} Let $\Theta\in\Con L$ and assume that  $X,Y,U,V,S,T$ are $\Theta$-blocks. Then 
\begin{align}
&X\vee Y=U \iff 0_X\vee 0_Y=0_U, \quad
X\wedge Y= V \iff 1_X\wedge 1_Y = 1_V, \label{eq:sDkvlDa}\\
&\text{whereby }\,\, S\leq T \iff  0_S\leq 0_T  \iff  1_S\leq 1_T,\label{eq:sDkvlDb}\\
& \text{and so } \kern8.5pt S=T \iff  0_S=0_T  \iff  1_S=1_T.\label{eq:sDkvlDc}
\end{align}
\end{lemma}

For an element $a$ of a lattice $L$, we use the notation
$\filter a=\filter_L a=:\set{x\in L: x\geq a}$ and $\ideal a=\ideal_L a:=\set{x\in L: x\leq a}$.
We need the following map (AKA function):
\begin{equation}
\faf\colon L\setminus\filter a\to\filter a,\text{ defined by }x\mapsto a\vee x.
\label{eq:faf}
\end{equation}

\begin{lemma}\label{lemma:dlatFlt} Let $L$ be a \emph{finite} distributive lattice.

\textup{(i)} If  $a\in\At L$,
then $\faf$ defined in \eqref{eq:faf} is a lattice embedding.

\textup{(ii)} If $a\in \At L$ has a complement, then $\faf$ is an isomorphism.

\textup{(iii)}  
$L$ is a boolean if and only if $\faf$ is an isomorphism for  (equivalently, if $\faf$ is bijective) for each atom $a$ of $L$.
\end{lemma}

\begin{proof} Clearly, $\faf$ is a lattice homomorphism by distributivity.
Assume that $b_1,b_2\in L\setminus\filter a$ such that
$\faf(b_1)=\faf(b_2)$. For $i\in\set{1,2}$, we have that $a\wedge b_i=0$ since $b_i\not\geq a \succ 0$. Hence, $b_i$ is a complement of $a$ in the interval $[0, a\vee b_1]=[0, a\vee b_2]$. But this interval is a distributive lattice, whereby the uniqueness of complements in distributive lattices imply that $b_1=b_2$. That is,  $\faf$ is injective, proving part (i).

Next, assume that $a\in \At L$ with a complement $a'$. Let $c\in\filter a$. Then $\faf(c\wedge a')=(c\wedge a')\vee a=(c\vee a)\wedge(a'\vee a)=c\wedge 1=c$. If we had that $c\wedge a'\in\filter a$, then $a\leq (c\wedge a')\wedge a=c\wedge (a'\wedge a)=c\wedge 0=0$ would be a contradiction. 
Hence, $c\wedge a'\in L\setminus\filter a$. Thus, $\faf$ is surjective and so it is an isomorphism, proving part (ii). 

The ``only if'' part of (iii) follows from part (ii). 
To prove the ``if'' part, assume that $L$ is a finite distributive lattice such that $\faf$ is bijective for every  $a\in\At L$. 
Let $\Jir L$ denote the poset  of (nonzero) join-irreducible elements of $L$, and let  $\downset {\Jir L}=(\downset{\Jir L},\cup,\cap)$ be the lattice of its down-sets.  By the well-known structure theorem of finite distributive lattices, see, e.g., Gr\"atzer~\cite[Theorem 107]{gratzerLTFbook},
\begin{equation}
L\cong \downset{\Jir L}.
\label{eq:wMrkgrnbR}
\end{equation}

We claim that $\Jir L$ is an antichain. Supposing the contrary, let $a,b\in \Jir L$ such that $a<b$. We can assume that $a\in\At L$ since otherwise we can replace it by an atom of $\ideal a$. The join-irreducibility of $b$ implies that  $b\neq \faf(x)$ for any $x\in L\setminus\filter a$, contradicting the bijectivity of $\faf$. Hence, $\Jir L$ is an antichain and $\downset{\Jir L}$ is the (boolean) powerset lattice.  By \eqref{eq:wMrkgrnbR}, $L$ is boolean, completing the proof.  
\end{proof}

\begin{proof}[Proof of Theorem \ref{thmmain}] 
As a convention for the whole proof, $\alpha$ always denotes an \emph{atom} of the congruence lattice of our $n$-element algebra or lattice.
We prove our statements by  induction on $n$. 
For $n\in\set{1,2}$, $\Con A=\Equ A$ and the statement is clear.  So let $n\geq 3$, and assume that all the three parts of the theorem hold for all algebras and lattices that have fewer than $n$ elements. Let $A=(A,F)$ be an $n$-element  congruence distributive algebra. For an atom $\alpha\in\At{\Con A}$, we define
\begin{equation}
\CA(A,\alpha):=\set{\Theta\in \Con A: \alpha\leq\Theta}
\text{ and } \CB(A,\alpha):=\Con A\setminus\CA(A,\alpha),
\label{eq:CACB}
\end{equation}
According to \eqref{eq:faf}, $\falf$ is the map $ \CB(A,\alpha)\to \CA(A,\alpha)$ defined by $\beta\mapsto \alpha\vee \beta$.

\begin{figure}[ht]
\vspace{-0.5em}
\centerline
{\includegraphics[width=\textwidth]{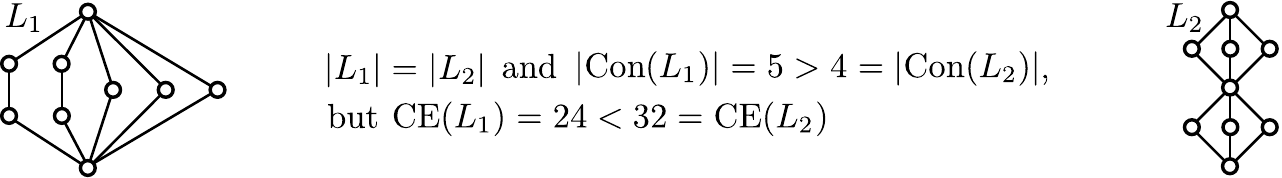}}      
\vspace{-0.8cm}
\caption{An example with two lattices}\label{figfadsom}
\label{fig:ex}
\vspace{-0.5em}
\end{figure}

For  $\beta\in\CA(A,\alpha)$,  we define $\beta/\alpha\in \Con{A/\alpha}\subseteq \Equ{A/\alpha}$ in the usual way: $\beta/\alpha:=\set{(x/\alpha,y/\alpha):
(x,y)\in\beta}$. By the
Correspondence Theorem, see, for example, Burris and Sankappanavar~\cite[Thm.\ 6.20]{burrsank},
$\CA(A,\alpha)\cong \Con{A/\alpha}$. Applying the Second Isomorphism Theorem, see, e.g., Burris and Sankappanavar~\cite[Theorem 6.15]{burrsank}, to the algebra $(A,\emptyset)$ with no operation, we obtain that 
$\nb{\beta/\alpha }{(A/\alpha)}
=
\nb{\beta} A$.
Hence, it follows from \eqref{eq:heqnb} that
$\heq(\beta) = n- \nb\beta A= n-\nb\alpha A+ \nb\alpha A-\nb{\beta/\alpha}{(A/\alpha)}=\heq(\alpha) + |A/\alpha|-\nb{\beta/\alpha}{A/\alpha}= \heq(\alpha) + \heq(\beta/\alpha)$. That is, 
\begin{equation}
\text{for every }\beta\in\CA(A,\alpha),\quad \heq(\beta)= \heq(\beta/\alpha) +  \heq(\alpha).
\label{eq:bcsshtdK}
\end{equation}
Since the  map $\CA(A,\alpha)\to \Con{A/\alpha}$, defined by $\beta\to \beta/\alpha$ is a lattice isomorphism by the Correspondence Theorem, we obtain by \eqref{eq:shmpwWr}, \eqref{eq:EnThdqm}, and  \eqref{eq:bcsshtdK} that 
\begin{align}
&|\CA(A,\alpha)| = |\Con{A/\alpha}|,
\label{eq:tlFlfTrtnz}
\\
&\En\beta=\En{\beta/\alpha}+\En\alpha\text{ for every }\beta\in\CA(A,\alpha),\text{ and }
\label{eq:wWrkntg}\\
&\sum_{\beta\in\CA(A,\alpha)}\En\beta=\CE{A/\alpha}+ \En\alpha\cdot|\Con{A/\alpha}|.
\label{eq:hvnkRbntTs}
\end{align}

Next, let $\gamma\in\CB(A,\alpha)$. Then $\gamma<\alpha\vee \gamma=\falf(\gamma)$ gives that $\heq(\gamma)<\heq(\falf(\gamma))$. Hence,  using \eqref{eq:EnThdqm} and the fact that  the function $\heq$ takes integer values, 
\begin{equation}
\text{for }\gamma\in \CB(A,\alpha), \quad  \heq(\gamma)\leq\heq(\falf(\gamma))-1\text{ and }\En{\gamma}\leq\En{\falf(\gamma)}-2.
\label{eq:gMsfkgrT}
\end{equation}
At $\leq'$ and $\leq^\ast$ below, we use  \eqref{eq:gMsfkgrT} and the injectivity of  $\falf$ (see Lemma~\ref{lemma:dlatFlt}), while we use
\eqref{eq:hvnkRbntTs} and $|\CA(A,\alpha)|=|\Con{A/\alpha}$ at $=^\dag$.
\begin{align}
\begin{aligned}
\sum_{\gamma\in\CB(A,\alpha)} \En\gamma &\leq'
\sum_{\gamma\in\CB(A,\alpha)} \bigl(\En{\falf(\gamma)}-2\bigr) 
\leq^\ast \sum_{\beta\in\CA(A,\alpha)} \bigl(\En\beta-2\bigr)\cr
&=-2|\CA(A,\alpha)|+ \sum_{\beta\in\CA(A,\alpha)} \En\beta\cr
&=^\dag  \CE{A/\alpha}+\bigl(\En \alpha -2 \bigr)\cdot |\Con{A/\alpha}|.  
\end{aligned}
\label{eq:ngnybrTsnKs}
\end{align}
It follows from \eqref{eq:hvnkRbntTs} and \eqref{eq:ngnybrTsnKs} that
\begin{equation}
\CE A \leq 2\cdot \CE{A/\alpha} +\bigl(2\cdot\En \alpha -2 \bigr)\cdot |\Con{A/\alpha}|.  
\label{eq:lkhmsvlTknL}
\end{equation}
Next,  we claim that
\begin{equation}\left.
\parbox{10cm}{if the inequality in \eqref{eq:lkhmsvlTknL} happens to be an equality, then $\falf$  is bijective and 
$\heq(\gamma) =\heq(\falf(\gamma))-1$ 
holds for every $\gamma\in\CB(A,\alpha)$; in particular, it holds for $\gamma=\botof A$ and so $\heq(\alpha)=1$.
}\,\,\right\}
\label{eq:ckfpCtrkT}
\end{equation}
To see this, note that  $\alpha=\falf(\botof A)$ is an $\falf$-image and 
$\En\beta-2=2(\heq(\beta)-1)>0$ 
for every $\beta\in\CA(A,\alpha)\setminus\set{\alpha}$.
Hence if $\falf$ was not surjective, then $\leq^\ast$ above would be a strict inequality and so would \eqref{eq:lkhmsvlTknL}. This yields that $\falf$ is surjective, whereby it is bijective by Lemma~\ref{lemma:dlatFlt}(i).
We know from \eqref{eq:EnThdqm} that the two inequalities  occurring in \eqref{eq:gMsfkgrT} are equivalent and so are the corresponding strict inequalities. So if $\heq(\gamma) =\heq(\falf(\gamma))-1$ failed for some $\gamma\in\CB(A,\alpha)$, then \eqref{eq:gMsfkgrT} would give that 
$\En\gamma< \En {\falf(\gamma)}-2$, whence $\leq'$ and  \eqref{eq:lkhmsvlTknL}  would be  strict inequalities, contradicting our assumption. Thus, we have verified \eqref{eq:ckfpCtrkT}. 
Next, we claim that
\begin{equation}\left.
\parbox{10.5cm}
{if $\Con A$ is distributive, $\alpha\in\At{\Con A}$ has a complement in $\Con A$, and $\heq(\alpha)=1$, then    $\CE A = 2\cdot \CE{A/\alpha} + 2\cdot |\Con{A/\alpha}|$.}
\,\,\right\}
\label{eq:tnvgbrmnTzh}
\end{equation}
To show \eqref{eq:tnvgbrmnTzh}, note that $\alpha$ is an atom of $\Equ A$ since $\heq(\alpha)=1$. Hence, by the semimodularity of $\Equ A$, 
$\falf(\gamma)=\alpha\vee_{\Con A}\gamma=\alpha\vee_{\Equ A}\gamma$ covers $\gamma$ in $\Equ A$. So $\heq(\gamma) =\heq(\falf(\gamma))-1$ and $\En\gamma =\En{\falf(\gamma)}-2$  for all $\gamma\in\CB(A,\alpha)$. Hence, $\leq'$ in \eqref{eq:ngnybrTsnKs} is an equality. 
So is $\leq^\ast$ in \eqref{eq:ngnybrTsnKs} since  $\falf$ is bijective by Lemma~\ref{lemma:dlatFlt}(ii). Thus,  both \eqref{eq:ngnybrTsnKs} and \eqref{eq:lkhmsvlTknL} are equalities, implying the validity of \eqref{eq:tnvgbrmnTzh}.

Next, we define an integer-valued function with domain $\set{4,5,6,7,\dots}$ as follows.
\begin{equation}
\parbox{9.5cm}{With the initial value $\gpn(4):=17/2$,  $\gpn(k)$ for $k\geq 5$ is given by the recursive formula
$\gpn(k):=2 \gpn(k-1)+5\cdot 2^{k-5}$.}
\label{eq:gpndF}
\end{equation}
The ``pentagon'' lattice $N_5$ is drawn in Figure~\ref{fig:M3N5}. The subscript of $\gpn$ comes from  ``\raisebox{.3ex}{P}e\raisebox{.3ex}{N}tagon''; this is motivated by the following claim, in which $k$ denotes an integer. 
\begin{equation}\left.
\parbox{10.5cm}{If a $k$-element lattice $K$ is of the form $K=C'\gplu N_5\gplu C''$ with chains $C'$ and $C''$, then $\CE K=\gpn(k)$ and $|\Con K|=5\cdot 2^{k-5}$.}\,\,\right\}
\label{eq:hmgYrndlt}
\end{equation}
We prove this by induction on $k$. If $k=5$, then $K\cong N_5$ 
and  Lemma~\ref{lemma:aboutcon} 
yields that $|\Con {N_5}|=5=5\cdot 2^{5-5}$ and  $\CE{N_5}=22=\gpn(5)$. Hence, \eqref{eq:hmgYrndlt} holds for $k=5$. So assume that $k>5$ and \eqref{eq:hmgYrndlt} holds for $k-1$. Since  $|C'|>1$ or $|C''|>1$, duality allows us to assume that $|C'|>1$. Then  $K$ has a unique atom $b$. By parts (2) and (3) of Lemma~\ref{lemma:aboutcon},
$\beta:=\equ(0,b)\in\At{\Con K}$, and  $[b,1]$ is the only non-singleton block of $\gamma:=\con(b,1)$. For $K^\dag:=K/\beta$, \eqref{eq:heqnb} gives that $|K^\dag|=|K|-\heq(\beta)=k-1$ and, in addition, $K^\dag =C_\dag'\gplu N_5^\dag\gplu C_\dag''$ where $C_\dag'$ and $C_\dag''$ are chains. 
Since $\gamma$ is a complement of $\beta$ and $\heq(\beta)=1$,  \eqref{eq:tnvgbrmnTzh} and the induction hypothesis imply that 
$\CE K=2\cdot\CE{K^\dag}+2|\Con {K^\dag}|=2\gpn(k-1)+2\cdot 5\cdot 2^{k-1-5}=2\gpn(k-1)+ 5\cdot 2^{k-5}=\gpn(k)$, as required. Furthermore, since $\fabe$ is bijective by Lemma~\ref{lemma:dlatFlt}(ii) and $|\filter_{\Con K}\,\beta|=|\Con {K^\dag}|=5\cdot 2^{k-6}$ by the Correspondence Theorem and the induction hypothesis, we have that 
$|\Con K|=2\cdot |\Con K^\dag|=5\cdot 2^{k-5}$. This completes the induction step and the proof of \eqref{eq:hmgYrndlt}.

If $\alpha\in\At{\Con A}$ is fixed and so no ambiguity threatens,  we let
\begin{equation}
m:=\heq(\alpha)=\En\alpha/2;\text{  note that }|A/\alpha|=n-m.
\label{eq:hQpsnTnMm}
\end{equation}
Equalities obtained by straightforward computations will be denoted by $\eqflat$ signs.
\begin{align}
\text{Let }\quad w(x)&:=\gmx(n)- \Bigl(2\cdot\gmx(n-x)+(2\cdot 2 x -2)\cdot 2^{n-x-1}\Bigr) \label{eq:f1aUxf} \\
&\phantom{:}\eqflat  2^{n-x} \cdot \bigl((n-1)\cdot 2^{x-1}-n-x+2\Bigr).
\label{eq:scpfSrthnd}
\end{align}
Keeping  $n\geq 3$ in mind, we claim that this auxiliary function has the  property that
\begin{equation}
\text{for  $1\leq x\leq n-2$, $\,\,w(x)\geq 0$ and $w(x)=0\iff x=1$.}
\label{eq:sMjLCkdsps}
\end{equation}
Let $w_2(x)$ denote the second factor of  \eqref{eq:scpfSrthnd}. It suffices to show that \eqref{eq:sMjLCkdsps} holds for $w_2(x)$ instead of $w(x)$. 
We denote $\frac d{dx} w_2(x)$ by $w_2'(x)$. 
Since $w_2(1)=0$ and $w_2'(x)
=\bigl((n-1)\cdot 2^{x-1}-n-x+2\Bigr){}'
\eqflat
(n-1)\cdot 2^{x-1}\cdot \ln 2-1 \geq 2\cdot \ln 2 \cdot 1-1=\ln 4-1>0$ implies that $w_2(x)$ is strictly increasing in the interval $[1,\infty)$, we conclude \eqref{eq:sMjLCkdsps}.

If $m=\heq(\alpha)=n-1$, then $A$ is a simple algebra and part (a) as well as parts (b) and (c) of the theorem are trivial. Hence, we can always assume that $m\leq n-2$. 
By the induction hypothesis, \eqref{eq:hQpsnTnMm}, and Lemma~\ref{lemma:manycon}, 
\begin{equation}
\CE{A/\alpha}\leq \gmx(n-m)\quad\text{ and }\quad |\Con{A/\alpha}|\leq 2^{n-m-1}.
\label{eq:mspmbRknfn}
\end{equation} 
Hence,  letting $m=\heq(\alpha)=\En\alpha/2$ play the role of $x$, we have that 
\begin{equation}
\CE A
\overset{\eqref{eq:lkhmsvlTknL}, \eqref{eq:mspmbRknfn}}\leq 
2\cdot\gmx(n-m)+(2\cdot 2 m -2)\cdot 2^{n-m-1} 
\overset{\eqref{eq:f1aUxf},\eqref{eq:sMjLCkdsps}}\leq \gmx(n),
\label{eq:zgGrmbRkszFlg}
\end{equation}
proving that $\CE A\leq \gmx(n)$, as required. Next, assume that $\CE A=\gmx(n)$. Then both inequalities in \eqref{eq:zgGrmbRkszFlg} are equalities, whereby the same holds for the inequalities in \eqref{eq:lkhmsvlTknL} and  
\eqref{eq:mspmbRknfn}, and $\heq(\alpha)=m=x=1$ by \eqref{eq:sMjLCkdsps}.
Note that \eqref{eq:ckfpCtrkT} also gives that $\heq(\alpha)=1$ and,  furthermore, it gives that $\falf$ is bijective. 
Since it is irrelevant how the atom $\alpha\in\Con A$ was fixed, 
\begin{equation}
\parbox{9cm}{for every  $\alpha\in\At{\Con A}$, $\falf$ is bijective and $\heq(\alpha)=1$.}
\label{eq:nzjtnmkcClt}
\end{equation}
Thus, Lemma~\ref{lemma:dlatFlt}(iii) implies that $\Con A$ is a boolean lattice. To show that this boolean lattice is of size $2^{n-1}$, we consider $\alpha$ fixed again. 
We have already mentioned that the inequalities in \eqref{eq:mspmbRknfn} are equalities, whence 
\eqref{eq:tlFlfTrtnz},  \eqref{eq:mspmbRknfn}, and the equality in \eqref{eq:nzjtnmkcClt} give that $\CA(A,\alpha)=2^{n-\heq(\alpha)-1}=2^{n-2}$. Thus, using that  $\falf$ is bijective, we obtain that 
$\Con A=2\cdot \CA(A,\alpha)=2^{n-1}$. Therefore, $\Con A$ is the $2^{n-1}$-element boolean lattice, and we have proved part (a) of the theorem.

Next, we turn our attention to part (b). The inequality in it follows from part (a) since lattices are congruence distributive. Let $L:=C_n$, the $n$-element chain, and 
let $u$ be the unique atom of $L$. It follows easily from Lemma~\ref{lemma:aboutcon} that $\alpha:=\equ(0,u)$ is an atom of $\Con{L}$. Hence, the chain $L':=L/\alpha$ is of size $|L'|=n-\heq(\alpha)=n-1$ by \eqref{eq:heqnb}. By Lemma~\ref{lemma:manycon}(b), $|\Con{L'}|=2^{n-2}$. 
Since $\Con L$ is boolean by Lemma~\ref{lemma:manycon} and $\heq(\alpha)=1$, \eqref{eq:tnvgbrmnTzh} gives that $\CE L=2\CE{L'}+2|\Con{L'}|$.  Using these facts and the induction hypothesis, we obtain that $\CE L$ is
\begin{equation}
2\gmx(n-1)+2\cdot 2^{n-2}=
2\bigl((n-2)\cdot 2^{n-2}+ 2^{n-2}\bigr)=\gmx(n),
\label{eq:ktmgltVwzfkrD}
\end{equation}
proving the ``if part'' of part (b).

Next, for later reference, we prove that
\begin{equation}\left.
\parbox{6cm}{if $\delta\in\Con L$ such that $\heq(\delta)=1$ and 
$L/\delta$ is a chain, then $L$ is also a chain.}\,\,\right\}
\label{eq:jvjkrsflGr}
\end{equation}
To prove \eqref{eq:jvjkrsflGr}, observe that  $L/\delta$
  consists of a unique 2-element $\delta$-block  $B=\set{0_B,1_B}$, and the rest of the $\delta$-blocks are singletons. Let $H:=\set{h}$ be a singleton $\delta$-block. Since $L/\delta$ is a chain, $B$ and $H$ are comparable; duality allows us to assume that $B<H$ holds in $L/\delta$. It follows from \eqref{eq:sDkvlDb} that $0_B<1_B\leq 1_H=h$. Hence, $h$ is comparable with the elements of $B$, and it is trivially comparable with  every element that forms a singleton block. Therefore, $L$ is a chain, proving \eqref{eq:jvjkrsflGr}. 

To prove the ``only if'' part of part (b), assume that $L$ is an $n$-element lattice and $\CE L=\gmx(n)$. By part (a) of the theorem, $\Con L$ is the $2^{n-1}$-element boolean lattice; let   $\alpha_1$,\dots, $\alpha_{n-1}$ be its atoms.
They are independent in the semimodular lattice $\Equ L$, whereby it is known, e.g. from Theorem 380 of Gr\"atzer~\cite{gratzerLTFbook}, that 
 $\heq(\alpha_1)+\dots+\heq(\alpha_{n-1})=\heq(\alpha_1\vee\dots\vee \alpha_{n-1})=\heq(\topof L)=n-1$. Hence each of the positive integers $\heq(\alpha_{1})$, \dots, $\heq(\alpha_{n-1})$ equals 1. In particular, letting $\alpha:=\alpha_1$, $\heq(\alpha)=1$. Thus $L':=L/\alpha$ is an $(n-1)$-element lattice by, say, \eqref{eq:heqnb}. By \eqref{eq:tnvgbrmnTzh}, 
\begin{equation}
2\cdot \CE{L'}+2\cdot|\Con{L'}|= \CE L = \gmx(n).
\label{eq:cphjMns}
\end{equation}
However, $\CE{L'}\leq \gmx(n-1)$ by part (a) of the theorem and 
$|\Con{L'}|\leq 2^{n-2}$ by  Lemma~\ref{lemma:manycon}(b). Hence, comparing \eqref{eq:ktmgltVwzfkrD} and \eqref{eq:cphjMns}, we obtain that $\CE{L'}=\gmx(n-1)$.  Thus, the induction hypothesis implies that $L'$ is a chain. By \eqref{eq:jvjkrsflGr}, so is $L$, proving  part (b) of the theorem. 

Next, note that $\gsb(k)$ is not an integer for an integer $k<3$. We claim  that 
\begin{align}
&\text{for }k\geq 3, \text{ }\quad \gsb(k)<\gmx(k),
\label{eq:zknjmszkSds}\\
&\text{for }k\geq 4, \text{ }\quad \gsb(k)\eqflat 2\gsb(k-1)+2^{k-2}, \text{ and}\label{eq:szkRshzTk}\\
&\text{for }k\geq 5, \text{ }\quad \gpn(k)<\gsb(k).
\label{eq:nljzsbrgpKh}
\end{align}
Indeed, \eqref{eq:zknjmszkSds} follows trivially from \eqref{eq:njpdFwmtS} while 
a trivial induction based on \eqref{eq:gpndF},  \eqref{eq:szkRshzTk}, and  $22=\gpn(5)=22<36=\gsb(5)$ and $5\cdot 2^{k-5} < 2^{k-2}$ yields \eqref{eq:nljzsbrgpKh}.

Next, we prove  part (c) of the theorem by induction on $n$. 
If $L$ is an $n$-element lattice such that $\CE L < \gmx(L)$, then part (b) of the theorem implies that $L$ is not a chain, whereby $n\geq 4$. So the base of the induction is $n=4$. 
For $n=4$, if $\CE L<\gmx(n)$, then $L=B_4$, the only 4-element non-chain, and  $\CE L=14=\gsb(4)$, whereby  part (c) of the theorem clearly holds for $n=4$. Thus, from now on,  we assume that $n\geq 5$ and $L$ is an $n$-element lattice such that $\CE L<\gmx(n)$ and part (c) of the theorem holds for all lattices consisting of fewer than $n$ elements.  By part (b), $L$ is not a chain. There are two cases.

\begin{case}\label{case:one} We assume that there is an  $\alpha\in\At{\Con L}$ such that $L':=L/\alpha$ is not a chain. 
For such an atom $\alpha$ and $m:=\heq(\alpha)=\En\alpha/2$, \eqref{eq:hQpsnTnMm} gives that $|L'|=n-m$. Hence $|\Con{L'}|\leq 2^{n-m-2}$ by Lemma~\ref{lemma:manycon}. By the induction hypothesis, $\CE{L'}\leq \gsb(n-m)$.
Thus, \eqref{eq:lkhmsvlTknL} yields that
\begin{equation}
\CE L\leq 2 \gsb(n-m) +(4m-2)\cdot 2^{n-m-2}.
\label{eq:tcpDhwjmG}
\end{equation}
This motivates us to consider the auxiliary function
\begin{equation}
u_n(x):=\gsb(n) - \bigl(2 \gsb(n-x) +(4x-2)\cdot 2^{n-x-2}\bigr),
\label{eq:wtgbszbcsmtV}
\end{equation}
where $x\in\mathbb R$ is a real variable. 
With the usual notation $u_n'(x):=\frac{d}{dx}u_n(x)$,
\begin{align}
u_n(x)&\eqflat (2n-1)\cdot 2^{n-3}-(4n+4x-6)\cdot 2^{n-x-3},
\label{eq:sKmtmrszvnyvgpTc}\\
u_n(1)&\eqflat 0\text{, \ and }\\
u_n'(x)&\eqflat \bigl((2(n+x)-3)\cdot\ln 4 -4\bigr) \cdot 2^{n-x-3}. 
\label{eq:smGzlkZgn}
\end{align}
Since $\ln 4>1$ and $n\geq 5$, for $x\in[1,\infty)$ we have that 
$\bigl(2(n+x)-3\bigr)\cdot\ln 4 -4\geq (2\cdot 6-3)\cdot 1-4=5>0$.
Hence, $u'_n(x)$ is positive and so $u_n(x)$ is strictly increasing in the interval $[1,\infty)$. Thus, for $x\geq 1$,  $u_n(x)\geq 0$ and $u_n(x)= 0\iff x=1$. Therefore, taking \eqref{eq:wtgbszbcsmtV} into account, 
\begin{equation}\left.
\parbox{8.5cm}{$2\cdot\gsb(n-m) +(4m-2)\cdot 2^{n-m-2}\leq \gsb(n)$, and this inequality turns to an equality if and only if $m=1$.}
\,\,\right\}
\label{eq:pXwnwmrnbD}
\end{equation} 
Combining \eqref{eq:tcpDhwjmG} and \eqref{eq:pXwnwmrnbD}, we obtain that $\CE L\leq \gsb(n)$, as required.

Next but still in the scope of Case~\ref{case:one}, assume that 
$\CE L = \gsb(n)$. Then \eqref{eq:tcpDhwjmG} and \eqref{eq:pXwnwmrnbD} give that $m=1$ and the inequality in \eqref{eq:tcpDhwjmG} is an equality. 
Since  \eqref{eq:tcpDhwjmG} was obtained from the inequalities  \eqref{eq:lkhmsvlTknL}, $|\Con{L'}|\leq 2^{n-m-2}$,  and   $\CE{L'}\leq\gsb(n-m)$, these three inequalities are also equalities. In particular, 
$\CE{L'}=\gsb(n-m)=\gsb(n-1)=\gsb(|L'|)$, and the induction hypothesis implies that $L'$ is of the form $L'=C^\ast\gplu B'_4\gplu C^{\ast\ast}$ where $C^\ast$ and $C^{\ast\ast}$ are finite chains and $B'_4$ is isomorphic to $B_4$. 
By Lemma~\ref{lemma:aboutcon}, there are $p,q\in L$ such that $p\prec q$ and $X:=\set{p,q}=[p,q]$ is the only non-singleton block of $\alpha$. Note that $p=0_X$ and $q=1_X$. Denote by $C'$ and $C''$ the sets $\set{y\in L: y/\alpha\in C^{\ast}}$ and $\set{y\in L: y/\alpha\in C^{\ast\ast}}$, respectively. 
Observe that $C'$ and $C''$ are chains. Indeed, if $x,y\in C'$, then 
either both $x/\alpha$ and $y/\alpha$ are singletons and their comparability in $C^\ast$ gives that $x\nparallel y$, or one of them is a singleton, the other one is $X=\set{p,q}$, and  \eqref{eq:sDkvlDb} yields that $x\nparallel y$.
Since  $C'$ and $C''$ are chains, we can assume that $X\in B'_4$ since otherwise $L=C'\gplu B_4\gplu C''$ is clear.
If $X$ is the bottom element of $B'_4$, then $B'_4$ is of the form $B'_4=\set{X,a/\alpha=\set a,b/\alpha=\set b,v/\alpha=\set v}$ with top element $\set v$, \eqref{eq:sDkvlDa} gives that $a\wedge b=1_X=q$, and we conclude  that  $\set{q,a,b,v}$ is sublattice of $L$, this sublattice is isomorphic to $B_4$, and  $L=C'\gplu B_4\gplu C''$ again, as required. 
By duality, $L$ is also of the required form $C'\gplu B_4\gplu C''$  if $X$ is the largest element of $B'_4$. 
We are left with the possibility that  
\begin{equation}
\text{$X\in B'_4$ is neither the bottom, nor the top of $B'_4$.}\label{eq:nzgtrmPnS}
\end{equation} 
Then $B'_4=\set{ \set{u}, \set{a}, X, \set{v}}$ such that $\set {u}$ and $\set{v}$ are the smallest element and the largest element of $B'_4$, respectively. Using \eqref{eq:sDkvlDa}, we have that $a\vee p=v$ and $a\wedge q=u$. Hence, $\set{u,a,p,q,v}$ is (isomorphic to) $N_5$; see Figure~\ref{fig:M3N5}. Using that $C'$ and $C''$ are chains, it follows that 
$L$ is of the form $L=C'\gplu N_5\gplu C''$. 
Hence, \eqref{eq:hmgYrndlt} and  \eqref{eq:nljzsbrgpKh}  yield that $\CE L=\gpn(n)<\gsb(n)$, contradicting our assumption. This excludes \eqref{eq:nzgtrmPnS} and completes Case~\ref{case:one} by having proved that 
\begin{equation}\left.
\parbox{10.5cm}{if $\CE L<\gmx(n)$ and $L/\alpha$ is not a chain for some $\alpha\in\At{\Con L}$, then $\CE L\leq\gsb(n)$ and, furthermore, $\CE L=\gsb(n)$ implies that $L=C'\gplu B_4\gplu C''$ for some chains $C'$ and $C''$.}
\,\,\right\} 
\label{eq:nvzTvKsnrP}
\end{equation}
\end{case}

\begin{case}\label{case:two} We assume that for every atom $\alpha\in\Con L$,  $L/\alpha$ is  a chain. 
Let $\alpha$ denote a fixed atom of $\Con L$. Similarly to the first part of Case~\ref{case:one} concluding with \eqref{eq:tcpDhwjmG} and using the same notation, $|L'|=n-m$, $|\Con{L'}| = 2^{n-m-1}$ by Lemma~\ref{lemma:manycon}(b), and $\CE{L'}\leq \gsb(n-m)$ by the induction hypothesis. Thus,  \eqref{eq:lkhmsvlTknL} yields that
\begin{equation}
\CE L\leq 2 \gsb(n-m) +(4m-2)\cdot 2^{n-m-1}.
\label{eq:nmgrMtsVztK}
\end{equation}
Since $L'$ is a chain but $L$ is not, \eqref{eq:jvjkrsflGr} implies that $m=\heq(\alpha)\geq 2$. Let
\begin{equation}
v_n(x):=\gsb(n) - \bigl(2 \gsb(n-x) +(4x-2)\cdot 2^{n-x-1}\bigr).
\label{eq:jbhdsmGnT}
\end{equation}
With this auxiliary real function, computation shows that
\begin{align}
v_n(x)&\eqflat 
(2n-1)\cdot 2^{n-3} -  (4n+12x-10)\cdot 2^{n-x-3},\cr
v_n(2)&\eqflat (2n-9)\cdot 2^{n-4} >  0,\,\,\text{ since $n\geq 5$, and}  
\label{eq:plTknktCrvNn}\\
v'_n(x)&\eqflat  ((4n+12x-10)\cdot \ln 2-12) \cdot  2^{n-x-3}.
\label{eq:tlrsKkrtktfld}
\end{align}
Since $n\geq 5$ and $x=m\geq 2$, we have that
$(4n+12x-10)\cdot \ln 2-12
\geq  34\cdot\ln 2 -12 =  17\cdot\ln 4 -12\geq 17-12>0.
$
Hence, $v'_n(x)>0$ and $v_n(x)$ is strictly increasing in 
$[2,\infty)$. This fact, $m\geq 2$, and \eqref{eq:plTknktCrvNn}
yield that $v_n(m)>0$. 
Therefore,  \eqref{eq:jbhdsmGnT} gives that 
$2\gsb(n-m) +(4m-2)\cdot 2^{n-m-1} < \gsb(n)$, whereby 
\eqref{eq:nmgrMtsVztK} implies that 
\begin{equation}\left.
\parbox{7.3cm}{if $\CE L<\gmx(n)$ and $L/\alpha$ is a chain for each $\alpha\in\At{\Con L}$, then $\CE L<\gsb(n)$,}
\,\,\right\}
\label{eq:whtwSczcpf}
\end{equation} 
completing the argument in Case~\ref{case:two}.
\end{case}

Next, we are going to prove by induction on $k=|K|$ that
\begin{equation}\left.
\parbox{8.3cm}{if $K$ is a $k$-element lattice of the form $C'\gplu B_4\gplu C''$ with chains $C'$ and $C''$, then $\CE K=\gsb(k)$.}\,\,\right\}
\label{pbx:whkjlCmgNzt}
\end{equation}
The smallest possible value of $k$ is 4, for which Lemma~\ref{lemma:aboutcon} yields easily that $\CE K=\CE{B_4}=14=\gsb(4)$. So let $k>4$. Duality allows us to assume that $|C'|\geq 2$ and $K$ has a unique atom $b$. 
Like in the argument proving \eqref{eq:hmgYrndlt}, 
$\gamma:=\con(b,1)$ is a complement of $\beta:=\equ(0,b)=\con(0,b)\in\At{\Con K}$ and $K/\beta$ is also of the form mentioned in \eqref{pbx:whkjlCmgNzt}. By Lemma~\ref{lemma:manycon}(c), $|\Con {K/\beta}|=2^{k-1-2}$. Thus, \eqref{eq:heqnb},
 \eqref{eq:tnvgbrmnTzh},  the induction hypothesis, and \eqref{eq:szkRshzTk} give that
\[
\CE K=2\CE{K/\beta} + 2|\Con{K/\beta}|=2\gsb(k-1)+2\cdot2^{k-1-2} =\gsb(K),
\]
proving \eqref{pbx:whkjlCmgNzt}.
Finally, \eqref{eq:nvzTvKsnrP}, \eqref{eq:whtwSczcpf}, and 
\eqref{pbx:whkjlCmgNzt} imply part (c) of the theorem. 
\end{proof}

\end{document}